\documentclass[preprint, 10pt, english]{amsart}
\usepackage{amsthm}
\usepackage{amsmath}
\usepackage{latexsym, amssymb}
\usepackage{txfonts}
\usepackage{mathtools}
\usepackage{color}
\usepackage[all]{xy}

\newtheorem{thm}{Theorem}[section] 

\newtheorem{cor}[thm]{Corollary}

\newtheorem{lem}[thm]{Lemma}
\newtheorem{prop}[thm]{Proposition}

\theoremstyle{definition}
\newtheorem{rem}[thm]{Remark}

\newcommand\operA[2]{{\if!#2!\operatorname{#1}\else{\operatorname{#1}_{#2}^{\phantom{I}}}\fi}} 

\def\Br{{\operatorname{Br}}}

\renewcommand{\H}{\operatorname H}

\def\Gal{{\operatorname{Gal}}}

\newcommand{\Trace}[1][]{\if!#1!\operatorname{Tr}\else{\operatorname{Tr}_{#1}^{\phantom{I}}}\fi} 

\long\def\forget#1\forgotten{{}} %

\def\({\left(}
\def\){\right)}

\newcommand\LAY[3][]{{\begin{array}{c}\mbox{#2} \if#1!{}\else{+}\fi \\ \mbox{#3}\end{array}}}

\makeatletter

\def\ps@pprintTitle{%
 \let\@oddhead\@empty
 \let\@evenhead\@empty
 \def\@oddfoot{}%
 \let\@evenfoot\@oddfoot}

\newcommand{\bigperp}{%
  \mathop{\mathpalette\bigp@rp\relax}%
  \displaylimits
}

\newcommand{\bigp@rp}[2]{%
  \vcenter{
    \m@th\hbox{\scalebox{\ifx#1\displaystyle2.1\else1.5\fi}{$#1\perp$}}
  }%
}
\makeatother

\renewcommand{\geq}{\geqslant}
\renewcommand{\leq}{\leqslant}

\DeclareMathOperator{\ed}{ed}

\newif\iffurther
\furtherfalse

\begin{document}

\title{Essential Dimension, Symbol Length and $p$-rank}

\author{Adam Chapman}
\address{Adam Chapman, Department of Computer Science, Tel-Hai College, Upper Galilee, 12208 Israel}
\email{adam1chapman@yahoo.com}
\author{Kelly McKinnie}
\address{Kelly McKinnie Department of Mathematical Sciences, University of Montana, Missoula, MT 59812, USA}
\email{kelly.mckinnie@mso.umt.edu}

\begin{abstract}
We prove that the essential dimension of central simple algebras of degree $p^{\ell m}$ and exponent $p^m$ over fields $F$ containing a base-field $k$ of characteristic $p$ is at least $\ell+1$ when $k$ is perfect.
We do this by observing that the $p$-rank of $F$ bounds the symbol length in $\operatorname{Br}_{p^m}(F)$ and that there exist indecomposable $p$-algebras of degree $p^{\ell m}$ and exponent $p^m$. We also prove that the symbol length of the Milne-Kato cohomology group $\operatorname H^{n+1}_{p^m}(F)$ is bounded from above by $\binom rn$ where $r$ is the $p$-rank of the field, and provide upper and lower bounds for the essential dimension of Brauer classes of a given symbol length.
\end{abstract}

\maketitle

\section{Introduction}

Given a field $k$ and a covariant functor $\mathcal{F} : \text{Fields}/k \rightarrow \text{Sets}$, the essential dimension of an object $x \in \mathcal{F}(F)$, denoted $\operatorname{ed}_{\mathcal{F}}(x)$, where $F$ is a field containing $k$, is the minimal transcendence degree of a field $E$ with $k \subseteq E \subseteq F$ for which there exists $x_0 \in \mathcal{F}(E)$ such that $x=x_0 \otimes_E F$. The essential dimension of the functor, denoted $\operatorname{ed}(\mathcal{F})$, is the supremum on the essential dimension of all the objects $x \in \mathcal{F}(F)$ for all fields $F \supseteq k$.
The essential $p$-dimension of an object $x \in \mathcal{F}$, denoted $\operatorname{ed}_{\mathcal{F}}(x;p)$, is defined to be the minimal $\operatorname{ed}_{\mathcal{F}}(x \otimes L)$ where $L$ ranges over all prime to $p$ field extensions of $F$. The essential $p$-dimension of $\mathcal{F}$, denoted $\operatorname{ed}(\mathcal{F};p)$, is defined to be the supremum on the essential $p$-dimension of all objects $x\in \mathcal{F}(F)$ for all fields $F \supseteq k$. Note that $\operatorname{ed}_{\mathcal{F}}(x;p) \leq \operatorname{ed}_{\mathcal{F}}(x)$ and $\operatorname{ed}(\mathcal{F};p) \leq \operatorname{ed}(\mathcal{F})$. See \cite{Merkurjev:2013} for a comprehensive discussion on these definitions and associated open problems.

Given a prime number $p$, a field $k$ of $\operatorname{char}(k)=p$ and integers $m,n$ with $n\geq m$, let $\operatorname{Alg}_{p^n,p^m}$ denote the functor mapping every field $F$ containing $k$ to the set of isomorphism classes of central simple algebras of degree $p^n$ and exponent dividing $p^m$ over $F$.
The computation of $\operatorname{ed}(\operatorname{Alg}_{p^n,p^m})$ is largely open (see \cite{BaekMerkurjev:2009, McKinnie:2017, Merkurjev:2013} for example).
It is known from \cite{Baek:2011} that for fields $k$ of $\operatorname{char}(k)=2$ we have $\operatorname{ed}(\operatorname{Alg}_{8,2}) \leq 10$, and when $k \supseteq \mathbb{F}_4$ we also have $\operatorname{ed}(\operatorname{Alg}_{4,2})=3$.
In the same paper it was also proven that for any $k$, $\operatorname{ed}(\operatorname{Alg}_{p^n,p^m};p) \geq 3$ when $n>m$.
In \cite{McKinnie:2017} it was shown that when $k$ is algebraically closed, $\ell+1 \leq\operatorname{ed}(\operatorname{Alg}_{p^\ell,p};p)$ for any $\ell$, improving the previous lower bound of 3 when $\ell \geq 2$.

The goal of this paper (Theorem \ref{Alg}) is to prove that $\ell+1 \leq \operatorname{ed}(\operatorname{Alg}_{p^{\ell m},p^m};p)$ when $k$ is perfect for any $\ell$ and $m$, recovering and extending the result from \cite{McKinnie:2017}. The techniques in this paper are simpler than \cite{McKinnie:2017}, relying on symbol length bounds and indecomposable division algebras which already exist in the literature, instead of analyzing sums of generic $p$-symbols. We also provide upper and lower bounds for the essential dimension of Brauer classes (and other Kato-Milne cohomology groups) of a given symbol length.

\section{Preliminaries}
For a field $F$ of $\operatorname{char}(F)=p$ and positive integers $m$ and $n$, the Kato-Milne coholology group $\H_{p^m}^{n+1}(F)$ is the additive group $W_m(F) \otimes \underbrace{F^\times \otimes \dots \otimes F^\times}_{n \ \text{times}}$ modulo the relations
\begin{itemize}
\item $(\omega^p-\omega) \otimes b_1 \otimes \dots \otimes b_n=0$, 
\item $(0\dots0,a,0,\ldots,0) \otimes a \otimes b_2 \otimes \dots \otimes b_n=0$, and 
\item $\omega \otimes b_1 \otimes \dots \otimes b_n=0$ where $b_i=b_j$ for some $i \neq j$,
\end{itemize}
where $W_m(F)$ is the ring of truncated Witt vectors of length $m$ over $F$, and for each $\omega=(\omega_1,\dots,\omega_m)$, $\omega^p$ stands for $(\omega_1^p,\dots,\omega_m^p)$. For a comprehensive reference on these groups see \cite{AravireJacobORyan:2018}.
The generators $\omega \otimes b_1 \otimes \dots \otimes b_n$ are called ``($p^m$-)symbols". For $n=1$, these groups describe the $p^m$-torsion of the Brauer group, i.e., $\operatorname{H}_{p^m}^2(F) \cong \operatorname{Br}_{p^m}(F)$ with the isomorphism given by $\omega \otimes b \mapsto [\omega,b)_F$, where $[\omega,b)_F$ stands for the cyclic algebra generated by $\theta_1,\dots,\theta_m$ and $y$ satisfying
$$\vec{\theta}^p-\vec{\theta}=\omega, \qquad y^{p^m}=b, \qquad \text{and} \qquad y\,\vec{\theta}\,y^{-1}=\vec{\theta}+\vec{1}$$
where $\vec{\theta}=(\theta_1,\theta_2,\dots,\theta_m)$ is a truncated Witt vector, $\vec{\theta}^p=(\theta_1^p,\theta_2^p,\dots,\theta_m^p)$, and $\vec{1}=(1,0,\dots,0)$ (see \cite{MammoneMerkurjev:1991} for reference). 
The symbol length of a class in $\H_{p^m}^{n+1}(F)$ is the minimal $t$ for which the class can be written as the sum of $t$ symbols. In the special case of symbol $p$-algebras of exponent dividing $p^m$, (i.e., $\operatorname H^2_{p^m}(F)$) the symbol length translates into the minimal $t$ for which the algebra is Brauer equivalent to a tensor product of $t$ cyclic algebras of degree $p^m$. We denote the symbol length of such a Brauer class $[A]$ by $\operatorname{sl}_{p^m}([A])$.

For any $t<m$, there is a shift map from the group $\operatorname{H}_{p^t}^{n+1}(F)$ to $\operatorname{H}_{p^m}^{n+1}(F)$ given by
$(a_1,\dots,a_t) \otimes b_1 \otimes \dots \otimes b_n \mapsto (0,\dots,0,a_1,\dots,a_t) \otimes b_1 \otimes \dots \otimes b_n$.
Since the map taking each symbol $\pi \in \operatorname{H}_{p^m}^{n+1}(F)$ to $\underbrace{\pi + \dots + \pi}_{p^{m-1} \ \text{times}}$ takes each $(a_1,\dots) \otimes b_1 \otimes \dots \otimes b_n$ to $(0,\dots,0,a_1^{p^m}) \otimes b_1 \otimes \dots \otimes b_n$ (which is equal to $(0,\dots,0,a_1) \otimes b_1 \otimes \dots \otimes b_n$ in this group), it gives rise to another map $\operatorname{Exp} : \operatorname{H}_{p^m}^{n+1}(F) \rightarrow \operatorname{H}_p^{n+1}(F)$.
\begin{thm}[{\cite[Theorem 2.31]{AravireJacobORyan:2018}; see also \cite[Theorem 1]{Izhboldin:2000} and \cite[Lemma 6.2]{Izhboldin:1996}}]\label{Exact}
The following sequence is exact
$$\xymatrix{
0\ar@{->}[r] & \H_{p^{m-1}}^{n+1}(F)\ar@{->}^{\operatorname{Shift}}[r] &  \H_{p^m}^{n+1}(F)\ar@{->}[r]^{\operatorname{Exp}} &  \H_{p}^{n+1}(F)\ar@{->}[r] & 0.}$$
\end{thm}
\section{Symbol Length and $p$-rank}

\begin{prop}\label{indecomposable}
Let $p$ be a prime integer, $k$ a field of $\operatorname{char}(k)=p$, and $m$ and $\ell$ positive integers with $\ell \geq 2$. Exclude the case of $p=\ell=2$ and $m=1$.
Then there exists a $p$-algebra $A$ of degree $p^{\ell m}$ and exponent $p^m$ over a field $F$ containing $k$ with $\operatorname{sl}_{p^m}([A_L]) \geq \ell+1$ for all prime to $p$ field extensions $L$ of $F$.
\end{prop}

\begin{proof}
If the symbol length of $A$ is at most $\ell$ then $A$ decomposes (as an algebra) as a tensor product of $\ell$ cyclic algebras of degree $p^m$. 
Hence we need only find a field $F\supset k$ and a $p$-algebra $A$ of degree $p^{\ell m}$ and exponent $p^m$ over $F$ which does not decompose as a tensor product of $\ell$ cyclic algebras of degree $p^m$ and retains this quality after any prime to $p$ extension. Such an algebra exists by \cite{Karpenko:1995} (see the introduction of that paper for explanation for $m>1$ and Sections 2 and 3 for $m=1$; the papers \cite{Karpenko:1998} and  \cite{McKinnie:2008} also discuss such examples).
\end{proof}

Every field $F$ of $\operatorname{char}(F)=p$ is a vector space over $F^p$, and if $[F:F^p]$ is finite then it is $p^r$ for some nonnegative integer $r$, called the ``$p$-rank of $F$", and denoted $\operatorname{rank}_p(F)$.

\begin{prop}[{cf. \cite[Remark 3.1]{ChapmanMcKinnie:2019}}]\label{prank}
Given a prime integer $p$, a field $F$ of $\operatorname{char}(F)=p$ and $\operatorname{rank}_p(F)=r < \infty$ and a $p$-algebra $A$ of exponent $p^m$ over $F$, we have $\operatorname{sl}_{p^m}([A]) \leq r$.
\end{prop}

\begin{proof}
Let $\alpha_1,\dots,\alpha_r$ be a $p$-basis for $F$, i.e., $F$ is spanned by $\{\alpha_1^{d_1} \dots \alpha_r^{d_r} : 0\leq d_1,\dots,d_r \leq p-1\}$ over $F^p$.
Since $F$ is isomorphic to $F^p$, by induction we get that $F$ is spanned by $\{\alpha_1^{d_1} \dots \alpha_r^{d_r} : 0\leq d_1,\dots,d_r \leq p^m-1\}$ over $F^{p^m}$.
Since $p^m$-symbols are split by $K=F[\sqrt[p^m]{\alpha_1},\dots,\sqrt[p^m]{\alpha_r}]$ (because $K^{p^m}=F$) and $\operatorname{Br}_{p^m}(F)$ is generated by $p^m$-symbols, every class in $\operatorname{Br}_{p^m}(F)$ is split by restriction to $K$ as well.
By \cite[Theorem 28]{Albert:1968} (for a more modern reference see \cite[Thm. 9.1.1]{GilleSzamuely:2017}), each class in $\operatorname{Br}_{p^m}(F)$ decomposes as a tensor product $C_1\otimes \dots \otimes C_r$ of $p^m$-symbols where each $C_i$ contains $F[\sqrt[p^m]{\alpha_i}]$, and therefore the symbol length of every class in $\operatorname{Br}_{p^m}(F)$ is at most $r$.
\end{proof}

In the remainder of the section, we provide another proof for Proposition \ref{prank}, and generalize it to higher cohomology groups.

\begin{lem} Given a $p$-basis $\alpha_1,\dots,\alpha_r$ for a field $F$ of $\operatorname{char}(F)=p$, $a \in W_m(F)$ and $b_1,\ldots, b_n \in F^\times$, there exist $w_i \in W_m(F)$ so that the following equality holds true in $\operatorname{H}_{p^m}^{n+1}(F)$:
\begin{equation*}
a\otimes b_1\otimes \cdots\otimes b_n =\sum_{\stackrel{i=(i_1,\ldots,i_n)}{1\leq i_1<\cdots<i_n\leq r}}w_i\otimes\alpha_{i_1}\otimes\cdots\otimes\alpha_{i_n}.
\end{equation*}
\end{lem}

\begin{proof}
We prove it by induction on $m$.
We know it holds for $m=1$ (see \cite[Remark 3.1]{ChapmanMcKinnie:2019}).
Suppose it holds for all positive integers smaller than $m$.
Take now a class $\pi$ in $\operatorname{H}_{p^m}^{n+1}(F)$.
The class $\operatorname{Exp}(\pi)$ lives in $\H_{p}^{n+1}(F)$, and so the statement holds true for it, i.e. $$\operatorname{Exp}(\pi)=\sum_{\stackrel{i=(i_1,\ldots,i_n)}{1\leq i_1<\cdots<i_n\leq r}}c_i\otimes\alpha_{i_1}\otimes\cdots\otimes\alpha_{i_n}$$ for some $c_i \in F$.
Then by Theorem \ref{Exact} the class $\pi$ differs from $$\sum_{\stackrel{i=(i_1,\ldots,i_n)}{1\leq i_1<\cdots<i_n\leq r}}(c_i,0,\dots,0)\otimes\alpha_{i_1}\otimes\cdots\otimes\alpha_{i_n}$$ by a class $\pi_0$ from the embedding of $\H_{p^{m-1}}^{n+1}(F)$ into $\H_{p^m}^{n+1}(F)$, and so the statement holds true also for $\pi_0$. Then $\pi$ is a sum of two classes for which the statements holds true, and by adding the Witt vector coefficients we see that it holds true also for $\pi$.
\end{proof}

\begin{cor} Let $F$ be a field of characteristic $p$ and finite $p$-rank $r$. Then the symbol length of a class in $\operatorname{H}^{n+1}_{p^m}(F)$ is at most $\binom rn$, and in particular the symbol length in $\Br_{p^m}(F)$ is at most $r$.
\end{cor}

\section{Other upper bounds on the Symbol Length}

As we saw, the $p$-rank provides a useful bound on the symbol length of classes in $\operatorname{H}_{p^m}^{n+1}(F)$.
However, in certain cases the $p$-rank can be infinite, and still one can obtain a reasonable finite upper bound on the symbol length.

\begin{prop}
Given finite symbol lengths $t$ of $\operatorname{H}_{p^m}^{n+1}(F)$ and $s$ of $\operatorname{H}_{p}^{n+1}(F)$, the symbol length of $\operatorname{H}_{p^{m+1}}^{n+1}(F)$ is at most $t+s$.
\end{prop}

\begin{proof}
Consider a class $\pi$ in $\operatorname{H}_{p^{m+1}}^{n+1}(F)$. The class $\operatorname{Exp}(\pi)$ lives in $\operatorname{H}_p^{n+1}(F)$, and therefore its symbol length is at most $s$, i.e. $\operatorname{Exp}(\pi)=a_1 \otimes b_1+\dots+a_s \otimes b_s$ for some $a_1,\dots,b_s \in F$. Then by Theorem \ref{Exact} $\pi$ differs from $(a_1,0\dots,0)\otimes b_1+\dots+(a_s,0,\dots,0)\otimes b_s$ by a class whose exponent divides $p^m$, so it belongs to the embedding of $\operatorname{H}_{p^m}^{n+1}(F)$ into $\operatorname{H}_{p^{m+1}}^{n+1}(F)$, and therefore its symbol length is at most $t$. Altogether, the symbol length of $\pi$ is at most $t+s$.
\end{proof}

\begin{cor}
Given a finite symbol length $s$ of $\operatorname{H}_{p}^{n+1}(F)$, the symbol length of $\operatorname{H}_{p^m}^{n+1}(F)$ is at most $m\cdot s$.
\end{cor}

\begin{proof}
By induction on $m$: if the symbol length of $\H_{p^m}^{n+1}(F)$ is bounded from above by $m\cdot s$, then by the previous proposition the symbol length of $\H_{p^{m+1}}^{n+1}(F)$ is bounded from above by $m\cdot s+s=(m+1)\cdot s$. 
\end{proof}

As proven in \cite{ChapmanMcKinnie:2019}, $\prod_{i=1}^n (\frac{u(F)}{2}-2^i+1)$ is an upper bound for the symbol length of $\operatorname{H}_{2}^{n+1}(F)$ when $u(F)<\infty$. Recall that $u(F)$ is the maximal dimension of a nonsingular anisotropic quadratic form over $F$, and it can be finite even when the 2-rank is infinite (see \cite{MammoneTignolWadsworth:1991}). As a result, we obtain an upper bound of $m\cdot \prod_{i=1}^n (\frac{u(F)}{2}-2^i+1)$ for the symbol length of $\operatorname{H}_{2^m}^{n+1}(F)$.
Similarly, over the $\widetilde{C}_{p,r}$-fields $F$ studied in \cite{ChapmanMcKinnie:2018}, for which the symbol length of $\H_p^2(F)$ was bounded from above by $p^{r-1}-1$, the symbol length of $\H_{p^m}^2(F)$ is therefore bounded from above by $m\cdot (p^{r-1}-1)$.

\section{The $\operatorname{Br}_{p^m}$ Functor}

Let $\operatorname{Br}_{p^m}$ denote the functor mapping each field $F$ containing $k$ to the group $\operatorname{Br}_{p^m}(F)$ of Brauer classes of central simple algebras of exponent dividing $p^m$ over $F$. For any central simple algebra $A$, let $[A]$ denote its Brauer class. 

\begin{rem}\label{r1}
For a given central simple algebra $A$ of degree $p^n$ and exponent $p^m$ over a field $F$ containing $k$, we have $\operatorname{ed}_{\operatorname{Br}_{p^m}}([A]) \leq \operatorname{ed}_{\operatorname{Alg}_{p^n,p^m}}(A)$ and $\operatorname{ed}_{\operatorname{Br}_{p^m}}([A];p) \leq \operatorname{ed}_{\operatorname{Alg}_{p^n,p^m}}(A;p)$. 
The reason is that for any algebra $A_0$ of degree $p^n$ and exponent $p^m $ over a field $F_0$ satisfying $k\subseteq F_0 \subseteq F$ and $A_0 \otimes F=A$, $[A]$ descends to $[A_0]$.
It is not necessarily an equality, because $[A]$ may descend to a $p^m$-torsion Brauer class $[B]$ where $B$ is of degree than $p^n$.
\end{rem}

Lemma \ref{symbol} below forms the basis for our lower bound on the essential dimension.

\begin{lem}[{\cite[Chapter V, Section 16.6, Corollary 3]{Bourbaki}}]\label{Jarden}
Suppose $F$ is a finitely generated field extension of transcendence
degree $t$ of a field $k$ of positive characteristic $p$ and $\operatorname{rank}_p(k)=r$. Then $\operatorname{rank}_p(F)=r+t$.\end{lem}

\begin{lem}\label{symbol}
Let $k$ be a field of $\operatorname{char}(k)=p>0$ and $\operatorname{rank}_p(k)=r$, and $A$ a $p$-algebra over a field $F$ containing $k$ of exponent $p^m$, $m\geq 1$.
Then $\operatorname{ed}_{\operatorname{Br}_{p^m}}([A])+r \geq \operatorname{sl}_{p^m}([A])$.
\end{lem}

\begin{proof}
Suppose the symbol length is $s$ and the essential dimension is $t$ where $r+t<s$.
Then $[A]$ descends to $[A_0]\in\Br_{p^m}(E)$ with $k \subseteq E \subseteq F$ and $E$ finitely generated of finite transcendence degree $t$ over $k$.
By Lemma \ref{Jarden} the $p$-rank of $E$ is $r+t$. By Proposition \ref{prank} the symbol length of $A_0$ is at most $r+t$, and so is the symbol length of $A$, contradiction. 
\end{proof}

\begin{thm}\label{Alg}
Let $k$ be a field of $\operatorname{char}(k)=p>0$ and $\operatorname{rank}_p(k)=r$, and $m$ and $\ell$ be positive integers with $\ell\geq 2$. 
Then $\operatorname{ed}(\operatorname{Alg}_{p^{\ell m},p^m};p) \geq \ell+1-r$. In particular, when $k$ is perfect, we have $\operatorname{ed}(\operatorname{Alg}_{p^{\ell m},p^m};p) \geq \ell+1$.
\end{thm}

\begin{proof}
The case of $p=n=2$ and $m=1$ is covered by \cite[Corollary 2.2]{Baek:2011}, so we exclude this case in the following discussion.
Suppose to the contrary that the essential $p$-dimension of this functor is at most $\ell-r$.
Let $A$ be a central simple algebra of degree $p^{\ell m}$ and exponent $p^m$ over a field $F$ containing $k$.
Then there exists a prime to $p$ extension $L/F$ such that the essential dimension of $[A\otimes L]$ is $\leq \ell-r$, and therefore by Lemma \ref{symbol} the symbol length of $[A \otimes L]$ is at most $\ell$.
However, by Proposition \ref{indecomposable} there exists an algebra $A$ in this category whose symbol length is at least $\ell+1$ under restriction to any prime to $p$ field extension $L$ of $F$, contradiction.
Therefore $\ed_{\operatorname{Br}_{p^m}}([A];p) \geq \ell+1-r$, which means $\ed_{\operatorname{Alg}_{p^{\ell m},p^m}}(A;p) \geq \ell+1-r$, and as a result $\operatorname{ed}(\operatorname{Alg}_{p^{\ell m},p^m};p) \geq \ell+1-r$. 
\end{proof}

Note that the special case of Theorem \ref{Alg} for $m=1$ and $k$ algebraically closed was proven in \cite{McKinnie:2017}. The main two advantages of our approach is that (1) it is much simpler, and (2) it holds true for any field $k$ of finite $p$-rank, not just algebraically closed fields.
If the essential dimension of $\operatorname{Alg}_{p^t,p^m}$ is under discussion, the lower bound of $\ell+1-r$ is obtained from Theorem \ref{Alg} by taking $\ell=\lfloor \frac{t}{m} \rfloor$ as a result of $\ed(\operatorname{Alg}_{p^t,p^m}) \geq \ed(\operatorname{Alg}_{p^{\ell m},p^m})$.

\section{The $\H_{p^m}^{n+1}$ Functor}

Given a field $k$ of $\operatorname{char}(k)=p$, consider the functor $\operatorname{H}_{p^m}^{n+1}$ mapping each field $F$ containing $k$ to the group $\operatorname{H}_{p^m}^{n+1}(F)$.

For $1\leq i \leq \ell$ let $x_{i,1},\dots,x_{i,m},y_{i,1},\dots,y_{i,n}$ be independent indeterminates over $k$ and set $F_{\ell,m,n} = k(x_{1,1},\ldots,y_{\ell,n})$ the rational function field over $k$ in $(m+n)\ell$ indeterminates. Set $A_{\ell,m,n} = \sum_{i=1}^\ell (x_{i,1},\ldots,x_{i,m}) \otimes y_{i,1} \otimes \ldots \otimes y_{i,n}$. This class is ``the generic sum of $\ell$ symbols in $\H_{p^m}^{n+1}(F_{\ell,m,n})$". Note that it depends on the choice of $k$.
The following theorem gives a lower bound for the essential dimension of this generic sum of symbols:

\begin{thm}\label{GenSum}
Given a prime integer $p$, an algebraically closed field $k$ of $\operatorname{char}(k)=p$ and integers $m,n,\ell\geq 1$, the generic sum $A_{\ell,m,n} = \sum_{i=1}^\ell (x_{i,1},\ldots,x_{i,m}) \otimes y_{i,1} \otimes \ldots \otimes y_{i,n}$ of $\ell$ symbols in $\operatorname{H}_{p^m}^{n+1}(F_{\ell,m,n})$ is of 
$\operatorname{ed}_{\operatorname{H}^{n+1}_{p^m}}(A_{\ell,m,n};p) \geq \ell+n$.
\end{thm}

\begin{proof}
\sloppy By \cite[Theorem 5.8]{McKinnie:2017} the generic sum $A_{\ell,1,n}$ of $\ell$ symbols in $\operatorname{H}_{p}^{n+1}(F_{\ell,1,n})$ has  $\operatorname{ed}_{\operatorname{H}_{p}^{n+1}}(A_{\ell,1,n};p) \geq \ell+n$.
Write $B=A_{\ell,1,n} \otimes F_{\ell,m,n}$.
By taking $k'$ to be the algebraic closure of $k(x_{i,j} : 1\leq i \leq \ell, 2 \leq j \leq m)$, the class $B'=B \otimes k'(x_{1,1},\dots,x_{\ell,1},y_{1,1},\dots,y_{\ell,n}) $ is the generic sum of $\ell$ symbols over $k'$. Note that \cite[Theorem 5.8]{McKinnie:2017} only requires $k$ to be algebraically closed, so one can consider the same functors over $k'$ as well (because it is also algebraically closed), and over $k'$ we have $\ed_{\H_{p}^{n+1}}(B';p) \geq \ell+n$, and therefore over $k$ we have $\ed_{\H_{p}^{n+1}}(B;p) \geq \ell+n$.
Now, the generic sum $A_{\ell,m,n}$ of $\ell$ symbols in $\operatorname{H}_{p^m}^{n+1}(F_{\ell,m,n})$ is a pre-image of $A_{\ell,1,n} \otimes F_{\ell,m,n}$ under $\operatorname{Exp}$. 
If for some prime to $p$ field extension $L$ of $F_{\ell,m,n}$, $A_{\ell,m,n} \otimes L$ descends to $\pi$ in $H_{p^m}^{n+1}(E)$ for a field $E$ of transcendence degree less than $\ell+n$ over $k$, then $B \otimes L$ descends to $\operatorname{Exp}(\pi)$ which is in $\operatorname{H}_p^{n+1}(E)$, contradiction.
\end{proof}

\begin{rem}
The special case of Theorem \ref{GenSum} for $n=1$ coincides with Theorem \ref{Alg} for $k$ algebraically closed. Note that the general case of Theorem \ref{GenSum} with $n\geq 1$ requires the use of \cite{McKinnie:2017} rather than the existence of indecomposable algebras as in Theorem \ref{Alg}.
\end{rem}

In the rest of the section, we present an upper bound for the symbol length of sums of $\ell$ symbols in $\operatorname{H}^{n+1}_{p^m}(F)$ (which can also be the generic sum).
For this we recall the definition of $\ed(G)$, for $G$ a finite group, to be the essential dimension of the functor $\mathcal F:\textrm{Fields}/k \to \textrm{Sets}$ sending each field $F$ containing $k$ to the set of Galois field extensions $K/F$ with Galois group $\Gal(K/F) = G$. In the following example of Ledet, we have changed “exponent $p^m$” to  “exponent dividing $p^m$”, a trivial generalization of the original statement.

\begin{lem}[{\cite[Example, pg. 162]{Ledet:2004}}] 
Let $m$ and $\ell$ be positive integers, $k$ a field of $\operatorname{char}(k) = p$ with $|k|\geq p^\ell$ and $G$ a $p$-group of exponent dividing $p^m$ minimally generated by $\ell$ elements. Then $\ed(G)\leq m$.
\label{Ledet}
\end{lem}

\begin{rem}\label{LedetCor}
There is a correspondence between finite subgroups of $W_m(F)/\wp(W_m(F))$ and Galois field extensions $K/F$ with a Galois $p$-group $G$ of exponent dividing $p^m$:
Each class in $W_m(F)/\wp(W_m(F))$ of exponent $p^t$ ($t\leq m$) is represented by a class $\omega$ coming from the embedding of $W_t(F)/\wp(W_t(F))$ into $W_m(F)/\wp(W_m(F))$. Written as an element $\omega=(\omega_1,\dots,\omega_t)$ in $W_t(F)/\wp(W_t(F))$, the corresponding field extension is $C/F$ where $C$ is generated over $F$ by $\theta_1,\dots,\theta_t$ subject to the relations 
$$(\theta_1,\dots,\theta_t)^p-(\theta_1,\dots,\theta_t)=\omega$$
thinking about $(\theta_1,\dots,\theta_t)$ as a Witt vector in $W_t(C)$ and $(\theta_1,\dots,\theta_t)^p=(\theta_1^p,\dots,\theta_t^p)$ (see the Preliminaries section, and also \cite[Chapter 8]{Jacobson} and \cite{MammoneMerkurjev:1991}).
The extension $K/F$ corresponding to $G$ is thus the field compositum of the field extensions corresponding to the classes $G$ consists of.
By Lemma \ref{Ledet} we thus conclude that given any finite subgroup $G'=\langle \vec x_1 ,\dots,\vec x_\ell  \rangle$ of $W_m(F)/\wp(W_m(F))$, there exists an intermediate field $k \subseteq E \subseteq F$ of $\mathrm{tr.deg}_k(E)\leq m$ and a subgroup of $W_m(E)/\wp(W_m(E))$ to which $G'$ descends, and in particular, there exist classes $\vec z_1,\dots,\vec z_\ell$ in $W_m(E)/\wp(W_m(E))$ for which $\vec z_i \otimes F\equiv\vec x_i \pmod{\wp(W_m(F))}$ for each $i \in \{1,\dots,\ell\}$.
\end{rem}

\begin{prop} Let $k$ be a field of $\operatorname{char}(k)=p$ with $|k|\geq p^\ell$, $F$ a field containing $k$ and let $\pi$ be the sum of $\ell$ symbols in $\H^{n+1}_{p^m}(F)$ . Then $\ed_{\operatorname{H}_{p^m}^{n+1}}(\pi)\leq m+\ell n$. 
\label{upper bound}
\end{prop}

\begin{proof} 
Let $\pi = \sum_{i=1}^\ell (x_{i,1},\ldots,x_{i,m}) \otimes y_{i,1} \otimes \ldots \otimes y_{i,n}$. By Remark \ref{LedetCor} there exists an intermediate field $k \subseteq E \subseteq F$ and $z_{i,j}\in E$ so that $\vec x_i = (x_{i,1},\ldots,x_{i,m})\in W_m(F)$ satisfies $\vec x_i \equiv \vec z_i \otimes F \pmod{\wp(W_m(F))}$ and $\mathrm{tr.deg}_k(E)\leq m$. Therefore, $\pi$ is defined over $E(y_{i,j} : 1\leq i \leq \ell, 1 \leq j \leq n)$. This field has transcendence degree at most $m+\ell n$. 
\end{proof}

\begin{cor}Let $k$ be a field of $\operatorname{char}(k)=p$ with $|k|\geq p^\ell$, $F$ a field containing $k$ and let $\pi$ be the sum of $\ell$ symbols in $\Br_{p^m}(F)$ . Then $\ed_{\Br_{p^m}}(\pi)\leq m+\ell$.
\end{cor}

\begin{proof} This is the case of $n=1$ in Proposition \ref{upper bound}.
\end{proof}

Combining what we have obtained so far, we can outline the bounds as follows:

\begin{cor}\label{Bounds}
For any positive integer $m$ and a class $[A]$ in $\operatorname{Br}_{p^m}(F)$ where $F \supseteq k$ and $k$ is an infinite perfect field of $\operatorname{char}(k)=p$, we have
$$\operatorname{sl}_{p^m}([A]) \leq \operatorname{ed}_{\operatorname{Br}_{p^m}}([A]) \leq \operatorname{sl}_{p^m}([A])+m.$$
\end{cor}

And in particular, we get the following elegant upper bounds for the essential dimension when $k$ is an infinite perfect field:
\begin{cor}
\
\begin{itemize}
\item If $A$ is a $p$-algebra of degree $p^n$ and exponent $p^m$ over $F$, then $\operatorname{ed}_{\operatorname{Br}_{p^m}}([A]) \leq p^n+m-1$.
\item If $p=2$ and $A$ is of degree 8 and exponent 2 over $F$, then $\operatorname{ed}
_{\operatorname{Br}_{2}}([A]) \leq 5$.
\end{itemize}
\end{cor}

\begin{proof}
By \cite{Florence:2013}, the symbol length of $A$ of degree $p^n$ and exponent $p$ is at most $p^n-1$.
By \cite{Rowen:1984}, the symbol length of $A$ of degree 8 and exponent 2 is at most 4.
\end{proof}

Note that 5 is much better than the upper bound of 10 for $\operatorname{ed}_{\operatorname{Alg}_{8,2}}(A)$ obtained in \cite{Baek:2011}, but one should note that $\operatorname{ed}_{\operatorname{Alg}_{8,2}}(A)\geq \operatorname{ed}_{\operatorname{Br}_2}([A])$ and it is not necessarily an equality.

\begin{rem}
We believe that when $k$ is algebraically closed, the lower bound in Corollary \ref{Bounds} can be improved to $\operatorname{sl}_{p^m}([A])+1 \leq \operatorname{ed}_{\operatorname{Br}_{p^m}}([A])$.
If this is indeed true, that would mean that over fields of transcendence degree $d$ over $k$, every algebra $A$ of exponent $p^m$ has index $\operatorname{ind}(A) \leq p^{\operatorname{sl}([A])}\leq p^{\operatorname{ed}([A])-1}\leq p^{d-1}$, proving an important well-known conjecture on the period-index problem for this particular case. Unfortunately, we do not have a proof yet.
\end{rem}

\section*{Appendix -- Standard Case}

We conclude with some comments on the case of an algebraically closed field $k$ of characteristic prime to $p$.
Given a field $E$ of transcendence degree $r$ over $k$, $E$ is a $C_r$ field, and so the symbol length of a central simple algebra $A$ in $\operatorname{Br}_{p^m}(E)$ of degree $p^{\ell m}$ is at most $m(p^{r-1}-1)$ by \cite[Theorem 8.2]{Matzri:2016}. Therefore, if we start with a central simple algebra $A$ of degree $p^{\ell m}$ and exponent $p^m$ over a field $F$ containing $k$ whose essential dimension is $r$, then $\operatorname{sl}_{p^m}([A]) \leq m(p^{r-1}-1)$ (a formula which was already implicitly obtained in \cite[Section 5]{Matzri:2016}). By solving for $r$, we obtain the formula
$$\operatorname{ed}_{\operatorname{Alg}_{p^{\ell m},p^m}}(A)\geq \operatorname{ed}_{\operatorname{Br}_{p^m}}([A])\geq  1+\log_p\left(\frac{\operatorname{sl}_{p^m}([A])}{m}+1\right).$$
Excluding the case of $p=m=\ell=2$, the existence of indecomposable algebras $A$ of degree $p^{\ell m}$ and exponent $p^m$ provides algebras $A$ of symbol length at least $\ell+1$, which gives the bound
$$\operatorname{ed}(\operatorname{Alg}_{p^{\ell m},p^m}) \geq 1+\log_p\left(\frac{\ell+1}{m}+1\right).$$
This lower bound is by no means as good as the known bounds in the literature (see \cite{BaekMerkurjev:2012}), but it is possible that this technique could lead to a better bound if we found a way of constructing indecomposable algebras of sufficiently large prescribed symbol length, or if the upper bound for the symbol length for $C_r$-fields from \cite[Theorem 8.2]{Matzri:2016} could be improved (for example to the conjectured bound of $r$).

\section*{Acknowledgements}

We thank Skip Garibaldi, Alexander Merkurjev and Jean-Pierre Tignol for their comments on the manuscript. We also thank the anonymous referee for the careful reading of the manuscript and the helpful comments.

\bibliographystyle{alpha}
\bibliography{bibfile}

\end{document}